\documentclass{amsart}

\usepackage[english]{babel}
\usepackage[T1]{fontenc}
\usepackage[latin1]{inputenc}
\usepackage{lmodern}
\usepackage{color}
\usepackage{a4}
\usepackage[T1]{fontenc}
\usepackage[all]{xy}
\usepackage{verbatim}
\usepackage{stmaryrd}
\usepackage{amsfonts}
\usepackage{amsmath,amssymb,amsthm}
\usepackage{enumerate}
\usepackage{xspace}
\usepackage{ulem}
\usepackage{euscript}
\usepackage{epsfig}
\usepackage{fancyhdr}
\usepackage{txfonts}
\usepackage{babel,indentfirst}

\makeatletter \@addtoreset{equation}{section}

\usepackage[all]{xy}

\newtheorem{theorem}{Theorem}[section]
\newtheorem{lemma}[theorem]{Lemma}
\newtheorem{problem}[theorem]{Problem}
\newtheorem{proposition}[theorem]{Proposition}
\newtheorem{example}[theorem]{Example}
\newtheorem{definition}[theorem]{Definition}
\newtheorem{corollary}[theorem]{Corollary}

\newtheorem{remark}[theorem]{Remark}

\newcommand{\B}{\mathcal{B}}

\newcommand{\Z}{\mathbb{Z}}
\newcommand{\C}{\mathbb{C}}

\newcommand{\N}{\mathbb{N}}

\def\shift{\operatorname{shift}}

\begin{document}

\title{The Mean Transform and The Mean Limit of an Operator}

\author{Fadil Chabbabi }
\address{Universit\'e  Lille, UFR de Math\'ematiques, Laboratoire CNRS-UMR 8524 P. Painlev\'e, 
59655 Villeneuve d'Ascq  Cedex, France}
\email{Fadil.Chabbabi@univ-lille3.fr}

\author{Ra\'ul E. Curto }
\address{Department of Mathematics, University of Iowa, Iowa City, Iowa 52242, USA}
\email{raul-curto@uiowa.edu}

\author{Mostafa Mbekhta }
\address{Universit\'e  Lille, UFR de Math\'ematiques, Laboratoire CNRS-UMR 8524 P. Painlev\'e, 
59655 Villeneuve d'Ascq  Cedex, France}
\email{Mostafa.Mbekhta@math.univ-lille1.fr}

\thanks{The first and third named authors were partially supported by Labex CEMPI (ANR-11-LABX-0007-01). \ The second named author was partially supported by NSF grant DMS-1302666.} 



\date{}

\begin{abstract}
Let $T\in\B(\mathcal{H})$ be a bounded linear operator on a Hilbert space $\mathcal{H}$, and let $T \equiv V|T|$ be the polar decomposition of $T$. \ The mean transform of $T$ is defined by $\widehat{T}:=\frac{1}{2}(V|T|+|T|V)$. \ In this paper we study the iterates of the mean transform and we define the mean limit of an operator as the limit (in the operator norm) of those iterates. \ We obtain new estimates for the numerical range and numerical radius of the mean transform in terms of the original operator. \ For the special class of unilateral weighted shifts we describe the precise relationship between the spectral radius and the mean limit, and obtain some sharp estimates.
\end{abstract}

\maketitle


\section{Introduction}
Let $\mathcal{H}$ be a Hilbert space equipped with its inner product $\langle .,.\rangle$. \ We denote by $\B(\mathcal{H})$  the algebra of all bounded linear operators on $\mathcal{H}$, and by $\mathcal{U}(\mathcal{H})$  the set of unitary operators on $\mathcal{H}$. \ For an arbitrary operator $T\in \B(\mathcal{H})$, we denote by $\mathcal{R}(T)$, $\mathcal{N}(T)$ and  $T^*$  the range,  the null subspace, and the operator adjoint  of $T$, respectively. The numerical range of $T$ is the set 
$$W(T):=\Big\{ \langle Tx,x\rangle ~~:~~ x\in \mathcal{H},~~ \|x\|=1\Big\},$$
 and the numerical radius of $T$ is defined as 
$$w(T):=\sup\big\{|\lambda|~~ : ~~\lambda\in W(T)\big\}.$$
We  refer  the reader to  \cite{sw} for more information about the numerical range and numerical radius.

We also let $\sigma(T)$ denote the spectrum of $T$, and $r(T)$ denote its spectral radius.
An operator $T$ is said to be quasinormal if it commutes with $T^*T$. \ For $0< p \le 1$, we say that $T$ is $p$-hyponormal if $(T^*T)^p\geq (TT^*)^p$. \ In the case when $p=1$ and $p=\frac{1}{2}$, the operator $T$ is called hyponormal and semi-hyponormal, respectively.
 
As usual, for $T\in \B(\mathcal{H})$ we denote the modulus of $T$ by $|T|:=(T^*T)^{1/2}$ and we shall always write, without further
mention, $T=V|T|$ to be the canonical polar decomposition  of $T$, where $V$ is the appropriate partial isometry satisfying  $\mathcal{N}(V)=\mathcal{N}(T)$. \ The Aluthge transform of $T$ was defined in \cite{alu} by 
$$
\Delta(T):=|T|^{\frac{1}{2}}V|T|^{\frac{1}{2}}.
$$
The mean transform of $T$, recently introduced in \cite{hyy}, is given as  
\begin{equation}
\widehat{T}:=\frac{1}{2}(V|T|+|T|V). \label{mean}
\end{equation}
We refer the reader to \cite{alu, kp3, kp2, kp1, hyy, jkp} for such operator transforms.

It is well known  that the quasinormal operators are exactly the fixed points of the Aluthge transform and of the mean transform (see \cite{hyy}). 
 The sequences of iterates of mean and Aluthge transforms of an operator $T$ are denoted  by $\widehat{T}^{(n)}$ and $\Delta^{(n)}(T)$ 
(respectively), with $\widehat{T}^{(0)}=\Delta^{(0)}(T)=T$,  $\widehat{T}^{(n+1)}=\widehat{\widehat{T}^{(n)}}$ and $\Delta^{(n+1)}(T):=\Delta(\Delta^{(n)}(T))$. \   A list of detailed and informative articles on the subject includes \cite{ams, kp1, hyy, jkp}.
 
In the next sections we establish some new properties of the mean transform. \ We also introduce the mean limit of an operator, and we study the mean transform iterates in some particular cases. \ Moreover, we obtain several new relations between the Aluthge and mean transforms for unilateral weighted shifts. \ For this class of operators, we obtain some sharp estimates for the spectral radius of the mean transform and the mean limit. 

\section{Some results about the mean transform }
 
 Contrary to what happens with the Aluthge transform, the mean transform does depend on the polar decomposition of the given operator. \ For example, consider
 $
T=
\begin{pmatrix}
0 & 1 \\
0 & 0
\end{pmatrix}
$ acting on $\C^2$.
The canonical polar decomposition of $T$ is $T=V|T|$, where $|T|=\sqrt{T^*T}=\begin{pmatrix}
0 & 0 \\
0 & 1
\end{pmatrix}$
and 
$V=\begin{pmatrix}
0 & 1 \\
0 & 0
\end{pmatrix}.$
On the other hand, we can also write $T=U_{max}|T|$, where $U_{max}=\begin{pmatrix}
0 & 1 \\
1 & 0
\end{pmatrix}$ is unitary. \ This is the so-called {\it maximal} polar decomposition of $T$, since the partial isometry is unitary. \ In this case,  
$$U_{max}|T|+|T|U_{max}=\begin{pmatrix}
0 & 1 \\
1 & 0
\end{pmatrix}\ne V|T|+|T|V=\begin{pmatrix}
0 & 1 \\
0 & 0
\end{pmatrix}, $$
which shows that the mean transform depends on the polar decomposition. \ In what follows, we will always use the canonical polar decomposition when dealing with the mean transform.

\begin{proposition}\label{th1} Let $T\in \B(\mathcal{H})$ be an arbitrary operator. \ Then we have 
$$\mathcal{N}(\widehat{T})=\mathcal{N}(T).$$
In particular $\widehat{T}=0$ if and only if $T=0$.
\end{proposition}

\begin{proof} Let $T=V|T|$ be the canonical polar decomposition of $T$, and let $x\in \mathcal{N}(T)=\mathcal{N}(V)$. \ Then $|T|x=Vx=0$ and thus $\widehat{T}x=\frac{1}{2}(V|T|x+|T|Vx)=0$. \ This show that 
$\mathcal{N}(T)\subseteq  \mathcal{N}(\widehat{T})$. 

Conversely, if $x\in \mathcal{N}(\widehat{T})$ then 
$$
V|T|x+|T|Vx=0,
$$
and hence $$|T|x+V^*|T|Vx=V^*(V|T|+|T|V)x=0.$$
It follows that 
$$
 \langle |T|x,x \rangle + \langle V^*|T|Vx,x \rangle = \langle |T|x+V^*|T|Vx,x \rangle =0.
$$
Since $|T|$ and $V^*|T|V$ are both positive, we obtain 
$$
\||T|^{\frac{1}{2}}x\|= \langle|T|x,x \rangle=0.
$$
As a consequence, $$\mathcal{N}(\widehat{T})\subseteq \mathcal{N}(|T|)=\mathcal{N}(T),$$
as desired. \end{proof}

\begin{theorem} \label{thm1} Let $T\in\B(\mathcal{H})$.  Then the following statements are equivalent. 
\begin{enumerate}[(i)]
\item $T$ is invertible.
\item $\widehat{T}$ is invertible and $\mathcal{R}(T)$ is closed. 
\end{enumerate}
\end{theorem}

\begin{proof} Let $T=V|T|$ be the canonical polar decomposition  of $T$. \ 
$(i)\Rightarrow (ii)$ \ We assume that $T$ is invertible; then $\mathcal{R}(T)=\mathcal{H}$ is closed. \ Moreover, $V$ is unitary and the operators $|T|$ and $V^*|T|V$ are positive and invertible. \ Hence $|T|>0$, $V^*|T|V>0$ and  $V^*\widehat{T}=\frac{1}{2}(|T|+V^*|T|V)>0$ is invertible. \ It follows that $\widehat{T}=\frac{1}{2} V(|T|+V^*|T|V)$ is also invertible.

$(ii) \Rightarrow (i)$ \ Assume now that $\widehat{T}$ is invertible. \ From 
Proposition \ref{th1}, $T$ is one-to-one, so $V$ is isometry, i.e. $V^*V=I$. \
 It follows that $V^*\widehat{T}=\frac{1}{2}(|T|+V^*|T|V)$ is also 
invertible, and its inverse is $(\widehat{T})^{-1}V$. Therefore, $V$ is 
unitary, since it maps $(\mathcal{N}(T)^{\perp}$ isometrically onto $\overline{\mathcal{R}(T)}$. \ Therefore, $\mathcal{N}(T)=\mathcal{N}(T^*)=\{0\}$. \ 
Since $\mathcal{R}(T)$ is closed,   $\mathcal{R}(T)=\overline{\mathcal
{R}(T)}= (\mathcal{N}(T^*))^{\perp}=\mathcal{H}$. \ Hence $T$ is invertible. 
\ This completes the proof.
\end{proof}

\begin{remark} In Theorem \ref{thm1} (ii), the condition ``$\mathcal{R}(T)$ is closed'' is required; without it, the reverse implication is false, as shown by the following example.
\end{remark}

\begin{example}
Let us denote by $(e_n)_{n\in \Z}$ the canonical basis of $\ell^2(\Z)$, and by $T:\ell^2 (\Z)\to \ell^2(\Z)$ the weighted bilateral shift defined  by $Te_n=\alpha_n e_{n+1}$ for all $n\in \Z$, where 

$$
\alpha_n = \left\{
    \begin{array}{ll}
        1 & \mbox{if $n$   even } \\
        \frac{1}{n^2} & \mbox{if $n$  odd } . 
    \end{array}
\right.
$$
The  mean transform  $\widehat{T}$ is also a weighted shift, and we have  $\widehat{T}e_n= \widehat{\alpha}_ne_{n+1}$ for $n\in \Z$, where $$\widehat{\alpha}_n=\dfrac{\alpha_n+\alpha_{n+1}}{2}=\left\{
    \begin{array}{ll}
        \dfrac{1+\frac{1}{(n+1)^2}}{2} & \mbox{if $n$  even } \\
        
       ~~ \dfrac{1+\frac{1}{n^2}}{2}& \mbox{if $n$  odd }. 
    \end{array}
\right. 
$$ 
Clearly, $Te_{2n+1}\underset{n\to \infty}{\longrightarrow 0}$, and therefore the operator $T$ is not invertible. \ On the other \newline hand, we have  $1 \geq \widehat{\alpha}_n\geq \frac{1}{2}$  for all $n\in \Z$, and from this it follows that $\widehat{T}$ is invertible.  \qed
\end{example}
\begin{remark} In general we have:

(1)  $\sigma(T) \ne \sigma(\widehat{T})$ (see \cite{hyy});

(2)  $(\widehat{T})^{-1} \ne \widehat{T^{-1}}$  (see  \cite{jkp}).

\end{remark}

\begin{proposition}\label{pr1} Let $T\in \B(\mathcal{H})$. \ Then the following properties hold.
 
 $(i)$ \ For all $\alpha \in \C$, $\widehat{(\alpha T)}=\alpha\widehat{T}$.
 
 $(ii)$ \ For every unitary or anti-unitary  operator $U : \mathcal{H} \to \mathcal{H}$, we have  
 $$\widehat{UTU^*}=U\widehat{T}U^*.$$
 
 $(iii)$ \ $\widehat{T}=I\iff T=I.$
\end{proposition}

\begin{proof} \ (i) \ Straightforward from (\ref{mean}). 

(ii) \ Let $T=V|T|$ be the polar decomposition of $T$ and let $U : \mathcal{H} \to  \mathcal{H} $ be a unitary  operator. \ First note that
$$ \vert UTU^*\vert = U\vert T\vert U^*,$$
and we therefore have 
$$UTU^*=UV\vert T\vert U^*=(UVU^*)(U\vert T\vert U^*)=\tilde{V}\vert UTU^*\vert,$$
where $\tilde{V}=UVU^*$. \ Observe that $\tilde{V}$ is a partial isometry and $\mathcal{N}(UTU^*) = \mathcal{N}(\tilde{V})$; it follows that $\tilde{V}\vert UTU^*\vert$ is the polar decomposition of $UTU^*$. This implies that
\begin{align*}
\widehat{(UTU^*)} &= \frac{1}{2}(\tilde{V}|UTU^*|+|UTU^*|\tilde{V}) \\
 &=\frac{1}{2}( UV|T|U^*+ U|T|VU^*)\\
   & = U\widehat{T}U^*.
\end{align*}
When $U$ is anti-unitary, the result is obtained in a similar fashion.\\
(iii) \ The implication $(\Leftarrow)$ is obvious, so we focus on $(\Rightarrow$). \ Assume that $\widehat{T}=I$; hence  $V^*=V^*\widehat{T}=\frac{1}{2}(|T|+V^*|T|V)$ is a positive partial isometry. \ In particular, $V=V^*$ is an orthogonal projection. \ On the other hand, still using $\widehat{T}=I$, we can use Proposition \ref{th1}(i) and conclude that $T$ is one-to-one. Then  $V$ is an isometry, so $V^*V=V^2=I$. Therefore $V=I$ and $|T|=\widehat{T}=I$.  
\end{proof}

\begin{lemma}({\bf Heinz inequality}, cf. \cite{hein}) \ Let $A,B,X\in\B(\mathcal{H})$ such that $A$ and $B$ are positive operators. \ Then  
$$
\big\|A^{\frac{1}{2}}XB^{\frac{1}{2}}\big\| \leq \big\|\dfrac{AX+XB}{2}\big\|.
$$
\end{lemma}

\begin{corollary} \label{corT} Let $T\in\B(\mathcal{H})$. \ Then 
$$
\big\|\Delta(T)\big\|\leq \big\|\widehat{T}\big\|\leq \big\|T\big\|.
$$
In particular, \; $r(T)\leq \big\|\widehat{T}\big\|$.
\end{corollary}

For partial isometries, we have the following result.

\begin{proposition} Let $V\in\B(\mathcal{H})$ be a partial isometry. \ Then
$$
\widehat{V}=\frac{1}{2}(I+V^*V)V.
$$
In particular, 
$$
\sigma(V)=\sigma(\widehat{V}).
$$ 
\end{proposition}

\begin{proof}
The modulus of $V$ is $|V|= V^*V$ and the polar decomposition of $V$ is $V=V(V^*V)$. Hence 
\begin{equation}
\widehat{V}=\frac{1}{2}(V+V^*VV)=\frac{1}{2}(I+V^*V)V. \label{eqnpi}
\end{equation}
Since $VV^*V=V$, it follows that 
\begin{equation}
\sigma(\widehat{V})\backslash\{0\}=\sigma(\frac{1}{2}(I+V^*V)V)\backslash\{0\}=\sigma(\frac{1}{2}V(I+V^*V))\backslash\{0\}=\sigma(V)\backslash\{0\}. \label{eqnpi2}
\end{equation}
Now observe that if $V$ is invertible then $V^*V=I$ and therefore $\widehat{V}=V$; it follows that $\widehat{V}$ is also invertible. \ Conversely, if $\widehat{V}$ is invertible then (\ref{eqnpi}) implies that $V$ is left invertible, that is, $V$ is an isometry. \ This means $V^*V=I$, and therefore $\widehat{V}=V$, and a fortiori $V$ is also invertible. \ This argument together with (\ref{eqnpi2}) establishes the equality of the spectra. 
\end{proof}

By Corollary \ref{corT}, $\|\widehat{T}\|\leq \|T\|$. \ As a consequence, the norm of the iterated mean transforms $( \|\widehat{T}^{(n)}\|)_{n\in\N}$ is a non-increasing sequence. \ Since it is bounded below by $0$, it converges; we denote the limit by $\ell(T)$. 

\begin{definition}
Let $T\in\B(\mathcal{H})$. \ The {\it mean limit} $\ell(T)$ is the limit in norm of the sequence of mean transform iterates; that is,  
$$\ell(T)=\lim_{n\to \infty} \|\widehat{T}^{(n)}\|=\inf_{n\in\N} \|\widehat{T}^{(n)}\|.$$
\end{definition}

\begin{remark} Let  $U\in\mathcal{U}(\mathcal{H})$ and $\alpha \in\C$. \ Then  
\begin{enumerate}[(i)]
\item $\ell(UTU^*)=\ell(T)$ and $\ell(\alpha T)=|\alpha |\ell(T)$. 
\item In the case when $T$ is quasinormal,  $\ell(T)=\|T\|=r(T)$. 
\item If $T^2=0$ then $\widehat{T}^{(n)}=\frac{1}{2^n} T$; as a consequence, $\ell(T)=0$.
\end{enumerate}
For the reader's convenience we provide a proof of (iii). \ Consider the canonical polar decomposition $T = V\left|T\right|$, and recall that $\overline{\mathcal{R}(T)}=\mathcal{R}(V)$. \ Since $T^2=0$ we must have $\mathcal{R}(T) \subseteq \mathcal{N}(T) = \mathcal{N}( \left|T\right|)$. \ It then follows that $\left|T\right|V=0$, which readily implies $\widehat{T}=\frac{1}{2}T$. \ The desired result is now clear. 
\end{remark}

It is now natural to formulate the following
\begin{problem}
For a general bounded linear operator $T \in \B(\mathcal{H})$, describe what $\ell(T)$ says about $T$.
\end{problem}

For $x,y \in \mathcal{H} \; (x \ne 0, y \ne 0)$, we denote by  $x\otimes y$ the  rank one operator defined by 
 $$
(x\otimes y)u := \langle u,y \rangle x \; (u\in \mathcal{H}).
$$  
The $\lambda$-Aluthge transform \cite{ams} of a rank one operator is given in \cite{chf} as follows: 
$$\Delta_\lambda (x\otimes y)= \dfrac{ \langle x,y \rangle}{\|y\|^2} y\otimes y.$$

In the following lemma,  we give the mean transform of this class of operators, and we show that the sequence of their mean iterates converges to the Aluthge transform. 
 
\begin{lemma}\label{l1} Let $x,y \in \mathcal{H}$ be two nonzero vectors, let $T:=x\otimes y$ be the rank one operator with range generated by $x$, and let $n\in\N$. \ Then the $n$-th iterate of $T$ is 
$$\widehat{T}^{(n)}=\frac{1}{2^n}\big(x+(2^n-1)\dfrac{ \langle x,y  \rangle}{\|y\|^2}y\big)\otimes y.$$
In particular, 
$ \widehat{T}^{(n)} \underset{n\to +\infty}{\longrightarrow} \Delta_\lambda(T)$ and $\ell(T)=| \langle x,y \rangle|= r(T).$
\end{lemma}

\begin{proof}
We first exhibit the mean transform of a rank one operator. \ A simple calculation yields  
 $$|T|=\sqrt{T^*T}=\frac{\|x\|}{\|y\|}(y\otimes y).$$
Let 
$$
V:=\dfrac{1}{\|x\| \|y\|} x\otimes y.
$$
We then have $\mathcal{N}(V)=\mathcal{N}(T)$ and  $V^*V=\dfrac{1}{ \|y\|^2} y\otimes y$ is an orthogonal projection.
 Hence $V$ is a partial isometry and 
$$
V|T|=\Big(\dfrac{1}{\|x\| \|y\|} (x\otimes y)\Big) \Big(\frac{\|x\|}{\|y\|}(y\otimes y)\Big)=T.
$$
Therefore $T=V|T|$ is the canonical polar decomposition of $T$. \ It follows that 
\begin{eqnarray*}
\widehat{T}&=&\frac{1}{2}(V|T|+|T|V)
\\&=&\frac{1}{2}\big(x\otimes y+\dfrac{\|x\|}{\|y\|}(y\otimes y)  \dfrac{1}{\|x\| \|y\|} (x\otimes y)\big)
\\&=&\frac{1}{2}\big(x\otimes y+\dfrac{ \langle x,y \rangle}{\|y\|^2}(y\otimes y)\big)\\&=&
\frac{1}{2}(x+\dfrac{ \langle x,y \rangle}{\|y\|^2}y)\otimes y.
 \end{eqnarray*}
 Now, by induction on $n$ the equality 
$$
\widehat{T}^{(n)}=\frac{1}{2^n}\big(x+(2^n-1)\dfrac{ \langle x,y \rangle}{\|y\|^2}y\big)\otimes y
$$
holds immediately. \ Moreover,
$$
\widehat{T}^{(n)} \underset{n\to +\infty}{\longrightarrow} \dfrac{\langle x,y\rangle}{\|y\|^2} y\otimes y=\Delta(T).
$$
In particular, 
$$
\ell(T)=|\langle x,y \rangle|= r(T).
$$
\end{proof}

To study the mean iterates of a large class of Hilbert space operators we will first need the following result.

\begin{lemma}\label{l3}
Let $T\in \B(\mathcal{H})$, with canonical polar decomposition $T =  V|T|$. \ The  following assertions are equivalent.
\begin{enumerate}
\item $\mathcal{N}(T^*)\subseteq \mathcal{N}(T)$.
\item $VV^*|T|=|T|VV^*=|T|$.
\item $V^*$ is quasinormal (i.e., $VV^*V^*=V^*$).
\end{enumerate}
In this case, we have 
 $$
\widehat{T}=\frac{1}{2}(V|T|+|T|V)=\frac{1}{2}V(|T|+V^*|T|V).
$$
\end{lemma}
\begin{proof} $(1)\Rightarrow (2)$ \  Suppose  that $(1)$ holds. Then
$$
  \mathcal{N}(VV^*) = \mathcal{N}(V^*) = \mathcal{N}(T^*)\subseteq  \mathcal{N}(T) = \mathcal{N}(\vert T\vert).
  $$
Hence, $\vert T\vert(I - VV^*) = 0$. Thus $VV^*|T|=|T|VV^*=|T|$.\\

\smallskip
$(2)\Rightarrow (3)$ \  Suppose  that $(2)$ holds. Then
$$
  \mathcal{N}(V^*) \subseteq \mathcal{N}(\vert T\vert) =  \mathcal{N}(V).
  $$
 It follows, $V(I - VV^*) = 0$. Hence $VVV^* = V$ and thus $V^* = VV^*V^*$.\\

\smallskip
$(3)\Rightarrow (2)$ \  Since $V^* = VV^*V^*$, we have $V = VVV^*$. Hence $ \mathcal{N}(V^*)\subseteq  \mathcal{N}(V)$ and 
 $$
  \mathcal{N}(T^*) = \mathcal{N}(V^*)\subseteq  \mathcal{N}(V)  =  \mathcal{N}(T).
   $$
   This completes the proof.
\end{proof}

We now state and prove one of our main results.

\begin{theorem}\label{th5}  Let $T\in\B(\mathcal{H})$  and suppose that  $\mathcal{N}(T^*)\subseteq\mathcal{N}(T)$. \ Let $T =  V|T|$ be the canonical polar decomposition of $T$, and let $n\in\N$. \ Then 
\begin{equation}\label{ei}
\widehat{T}^{(n)}=\frac{1}{2^n}V\Big(\overset{j=n}{\underset{j=0}{\sum}} (^n_j)(V^*)^j|T|V^j\Big).
\end{equation}
\end{theorem}

\begin{proof}
We will use induction on $n$. \ For $n=0$, the equality (\ref{ei}) holds immediately. \ Since
 $\mathcal{N}(T^*)\subseteq\mathcal{N}(T)$, we can use Lemma \ref{l3} to conclude that  
 $$
\widehat{T}=\frac{1}{2}(V|T|+|T|V)=\frac{1}{2}V(|T|+V^*|T|V).
$$
In particular, (\ref{ei}) holds also for $n=1$. 

We now assume that (\ref{ei}) holds for $n\in\N, \; n\ge1$. \ From Proposition \ref{th1} we have  
$$
\mathcal{N}(V)=\mathcal{N}(T)=\mathcal{N}(\widehat{T}^{(n)}).
$$
Since $V^*VV^*=V^*$, it follows that 
$$
|\widehat{T}^{(n)}|=\frac{1}{2^n}\Big(\overset{j=n}{\underset{j=0}{\sum}} (^n_j)(V^*)^j|T|V^j\Big).
$$
Hence $\widehat{T}^{(n)}=V|\widehat{T}^{(n)}|$ is the canonical polar decomposition of $\widehat{T}^{(n)}$. \ Thus
\begin{eqnarray*}
\widehat{T}^{(n+1)}&=& \dfrac{\widehat{T}^{(n)}+|\widehat{T}^{(n)}|V}{2}\\
&=&\frac{1 }{2^{n+1}} \Big(V\overset{j=n}{\underset{j=0}{\sum}} (^n_j)(V^*)^j|T|V^j+\Big(\overset{j=n}{\underset{j=0}{\sum}} 
(^n_j)(V^*)^j|T|V^j\Big)V\Big)\\
&=&\frac{1}{2^{n+1}}V\overset{j=n}{\underset{j=0}{\sum}} (^n_j)\Big((V^*)^j|T|V^j+(V^*)^{j+1}|T|V^{j+1}\Big )\\
&=&\frac{1}{2^{n+1}} V\overset{j=n+1}{\underset{j=0}{\sum}} (^{n+1}_j) (V^*)^j|T|V^j.
\end{eqnarray*}
Hence (\ref{ei}) holds for $n+1$. \ This completes the proof.
\end{proof}

\begin{corollary}Let $T\in\B(\mathcal{H})$ be such that $T$ and $T^*$ are one-to-one. \
Then $\|\widehat{T}^{(n)}\|=\|\widehat{T^*}^{(n)}\|$ for all $n\in\N$. \ In particular, $T$ and $T^*$ have the same mean limit
$$
\ell(T)=\ell(T^*).
$$ 
\end{corollary}

\begin{theorem}Let $T\in\B(\mathcal{H})$ and suppose that $\mathcal{N}(T^*)\subseteq\mathcal{N}(T)$. \ If $T$ is a semi-hyponormal operator then $\widehat{T}$ is also semi-hyponormal. \ Moreover, the sequence of mean iterates converges in the strong operator topology to a normal operator $L\in\B(\mathcal{H})$, and we have 
$$
\ell(T)=\|L\|=\|\widehat{T}^{(n)}\|\;\;\text{ for all} \;\;n\in \N.
$$
\end{theorem}

\begin{proof}
Let $T = V|T|$ be the  polar decomposition of $T$. \  It is easy to get  that $\; |T^*|=V|T|V^*$. 
Suppose that $T$ is semi-hyponormal. Then 
$|T^*|=V|T|V^*\leq |T|$. \ Multiplying this inequality by $V^*$ on the left and by $V$ on the right,  we get that 
\begin{equation}\label{ei1}
 V|T|V^*\leq |T|\leq V^*|T|V.
\end{equation}
On the other hand,  since  $\mathcal{N}(T^*)\subseteq\mathcal{N}(T)$, Lemma \ref{l3} and a simple calculation  yield
$$
|\widehat{T}|=\frac{1}{2}(|T|+V^*|T|V).
$$
Hence, it follows from Lemma \ref{l3}  again that  $\widehat{T}=V|\widehat{T}|$ is the canonical polar decomposition of $\widehat{T}$, and 
\begin{eqnarray*}
|(\widehat{T})^*|&=&V|\widehat{T}|V^*
\\&=&\frac{1}{2}(V|T|V^*+VV^*|T|VV^*)
\\&=&\frac{1}{2}(V|T|V^*+|T|),\;\;\;\;(\text{since  $VV^*|T|=|T|$})
\\&\leq& \frac{1}{2}(|T|+V^*|T|V)=|\widehat{T}|\;\;\;(\text{see (\ref{ei1})}).
\end{eqnarray*}
This shows that $\widehat{T}$ is semi-hyponormal. 

\medskip

Since  $\mathcal{N}(T^*)\subseteq\mathcal{N}(T)$, we have  
$$
\mathcal{R}(\widehat{T})\subseteq \overline{\mathcal{R}(T)} \quad \text{and}    \quad      \mathcal{N}(\widehat{T}^*)\subseteq  \mathcal{N}(\widehat{T}).
$$
For the first inclusion, note that the condition on the kernels  implies that $\mathcal{R}(T^*)\subseteq \overline{\mathcal{R}(T)}$. \ It follows, by definition of $\widehat{T}$, that
$\mathcal{R}(\widehat{T})\subseteq \overline{\mathcal{R}(T)}$. The second inclusion is obtained as follows:
\begin{eqnarray*}
\widehat{T}^* x = 0 &\Rightarrow& T^*x + V^*|T|x = 0
\\&\Rightarrow&V|T|V^*x + VV^*|T|x = 0
\\&\Rightarrow&V|T|V^*x + |T|x = 0,\;\;\;\;(\text{since  $VV^*|T|=|T|$})
\\&\Rightarrow& x \in \mathcal{N}(|T|) = \mathcal{N}(T)=  \mathcal{N}(\widehat{T}).
\end{eqnarray*}

Now, by the induction we obtain, for all $n\in \N$,

$$\mathcal{R}(\widehat{T}^{(n)})\subseteq \overline{\mathcal{R}(T)},  \quad   \mathcal{N}((\widehat{T}^{(n)})^*)\subseteq  \mathcal{N}(\widehat{T}^{(n)}) \quad \text{and} \quad  \widehat{T}^{(n)}\; \; \text{is  semi-hyponormal}.$$ 

\ We also know that $\widehat{T}^{(n)}=V|\widehat{T}^{(n)}|$ is the canonical polar decomposition of $\widehat{T}^{(n)}$, with  
\begin{eqnarray}\label{eii}
|\widehat{T}^{(n+1)}|&=&\frac{1}{2}\big(|\widehat{T}^{(n)}|+V^*|\widehat{T}^{(n)}|V\big)
\geq |\widehat{T}^{(n)}|\;\;\;\big(\text{by (\ref{ei1})}\big).
\end{eqnarray}
In particular,  $(|\widehat{T}^{(n)}|)_{n\in\N}$ is an increasing sequence, so it converges in the strong operator topology to a positive operator $A\in\B(\mathcal{H})$ with $\mathcal{R}(A)\subseteq \overline{\mathcal{R}(T)}$. \ It follows that $VV^*A=A$. \ From (\ref{eii}), $A$ satisfies 
$A=\frac{1}{2}(A+V^*AV)$, and therefore $VA=AV$. \ It follows that 
$\widehat{T}^{(n)}=V|\widehat{T}^{(n)}|$ strongly converges to the normal operator $L:=VA=AV$. \ Again, from (\ref{eii}) we obtain $\|\widehat{T}^{(n+1)}\|=\|\widehat{T}^{(n)}\|=\|T\|=\|L\|$, for all $n\in\N$, as desired.
\end{proof}

\begin{remark} Under the assumption of Theorem \ref{th4}, if the sequence $\big((V^*)^k|T|V^k\big)_{k\in\N}$ converges in the strong (resp. weak) operator topology to an operator $A\in\B(\mathcal{H})$, then so does the sequence of operator mean iterates $(\widehat{T}^{(n)})_{n\in\N}$; the limit is the normal operator $L=VA=AV$.
\end{remark}

\section{Numerical range and numerical radius }

\begin{theorem}\label{th4} Let $T\in\B(\mathcal{H})$. \ Then
$$\overline{W(\widehat{T})}\subseteq \overline{W({T})},$$
where $ \overline{W({T})}$ denotes the closure of the
numerical range $W(T)$ of $T$. \ In particular, 
$$w(\widehat{T})\leq w(T).$$
\end{theorem}

\begin{proof}
Recall first the well known formula for the numerical range, (see \cite[Theorem 4 and Corollary]{sw})
$$
\overline{W(T)}=\bigcap_{\lambda \in \C}\Big\{~~\mu\; : \; |\mu-\lambda|\leq \|T-\lambda I\|~~\Big\}.
$$
Let $T=V|T|$ be the canonical polar decomposition of $T$. \ From \cite[Lemma 2.3]{fjkp} we have 
$$ \|\; |T|V-\lambda I \, \|\leq \|T-\lambda I  \|,\;\; (\text{all }\;\; \lambda \in \C).
$$ 
Therefore,  
\begin{eqnarray*}
\overline{W(\widehat{T})}&=&\bigcap_{\lambda \in \C}\Big\{~~\mu\; : \; |\mu-\lambda|\leq \|\widehat{T}-\lambda I\|~~\Big\}
\\&=&\bigcap_{\lambda \in \C}\Big\{~~\mu\; : \; |\mu-\lambda|\leq \frac{1}{2}\|T-\lambda I+|T|V-\lambda I\|~~\Big\}
\\&\subseteq & \bigcap_{\lambda \in \C}\Big\{~~\mu\; : \; |\mu-\lambda|\leq \frac{1}{2}(\|T-\lambda I\|+\||T|V-\lambda I\|)~~\Big\}
\\&\subseteq & \bigcap_{\lambda \in \C}\Big\{~~\mu\; : \; |\mu-\lambda|\leq \|T-\lambda I\|~~\Big\}
=\overline{W(T)}.
\end{eqnarray*}
\end{proof}

\begin{lemma}\label{lemmsy}(Cf. \cite{msy}) \ Let $A,B, X\in\B(\mathcal{H})$ such that $A,B$ are positive. Then 
$$w\big(A^{\frac{1}{2}}X B^{\frac{1}{2}}\big)\leq w\Big(\dfrac{A^\alpha XB^{1-\alpha}+A^{1-\alpha} XB^{\alpha}}{2}\Big).$$
for all $0\leq \alpha\leq 1$.
\end{lemma}

As a direct consequence of Theorem \ref{th4} and Lemma \ref{lemmsy} we get the following result.
 
\begin{corollary}Let $T\in \B(\mathcal{H})$ be an arbitrary operator, and recall that $\Delta(T)$ denotes the Aluthge transform of $T$. \ Then 
$$w(\Delta(T))\leq w(\widehat{T})\leq w(T).$$
\end{corollary}

\section{The mean limit for unilateral weighted shifts}

Let $\ell^2(\N)$ be the  Hilbert space of complex square-summable sequences $x=(x_i)_{i\in\N}$, with the norm $\|x\|:=(\sum_{i\in\N}|x_i|^2)^{\frac{1}{2}}$. \ Given any bounded sequence of strictly positive numbers $\alpha\equiv(\alpha_i)_{i\in\N}$, the
associated unilateral weighted shift $W_\alpha\equiv \shift  (\alpha_0,\alpha_1...): \ell^2(\N)\longrightarrow \ell^2(\N)$ is defined by 
$$
W_\alpha (x_0,x_1,...):=(0,\alpha_0 x_0,\alpha_1 x_1,...),
$$
where $x=(x_i)_{i\in\N}\in \ell^2(\N)$. \ When $\alpha_i=1$ for all $i$, $W_\alpha=W$ is simply the standard  (unweighted) unilateral shift on $\ell^2(\N)$.
 
Clearly, $W_\alpha$ is a bounded linear operator on $\ell^2(\N)$, with operator norm $\|W_\alpha\|=\underset{i\in\N}{\sup}~~ \alpha_i$. \ The spectral radius of $W_\alpha$ is well known (see, for example, \cite[Problem 91]{Halmos}): 
\begin{equation}\label{eqq}
r(W_\alpha)=\underset{n\rightarrow \infty}{\lim}\Big(\underset{k\in\N}{\sup}(\alpha_k...\alpha_{k+n-1})\Big)^{\frac{1}{n}}.
\end{equation}
The spectrum of $W_\alpha$ is given in \cite[p.66, Theorem 4]{Shields} by 
$$
\sigma(W_\alpha)=\{ \lambda \in\C : |\lambda|\leq r(W_\alpha)\}.
$$
The Aluthge transform $\Delta(W_\alpha)$ of $W_\alpha$ is also a unilateral weighted shift:
$$
\Delta(W_\alpha)=W_{\alpha^{\Delta}}=\shift  \Big (\sqrt{\alpha_0\alpha_{1}}, \sqrt{\alpha_1\alpha_{2}}, ... \Big).
$$
By induction, the iterates of the Aluthge transform are given in \cite{kp1} by 
 $$\Delta^{(n)}(W_\alpha)=W_{\alpha^{\Delta^{(n)}}},$$
 where 
 $$\alpha^{\Delta^{(n)}}=\Big\{\alpha^{\Delta^{(n)}}_i\Big\}_{i\in\N}=
 \Big\{\Big(\prod_{j=0}^{n}\alpha_{i+j} ^{(_j ^n)} \Big)^{\frac{1}{2^n}}\Big \}_{i\in \N}$$
 and  $(_j ^n)=\dfrac{n!}{j!(n-j)!}$.
 
  As explained in \cite{hyy}, the mean transform $\widehat{W_\alpha}$ of $W_\alpha$ is 
 $$\widehat{W_\alpha}=W_{\widehat{\alpha}}=\shift  \Big (\frac{\alpha_0+\alpha_1}{2},\frac{\alpha_1+\alpha_2}{2},...\Big).$$
 The mean iterates of the weighted shift $\widehat{W_\alpha}^{(n)}$ are also weighted shifts with weight sequences 
\begin{equation}\label{eqm}
{\alpha}^{(n)}=\big\{\widehat{\alpha}_i^{(n)}\big\}_ {i\in\N } =\Big\{\dfrac{\sum_{i=0}^n(^n_j)\alpha_{i+j}}{2^n}\Big\}_{i\in\N}.
\end{equation}
 
 We remark that,  for a sequence of strictly positive numbers  $\alpha=(\alpha_i)_{i\in\N}$, we have the following relation between the iterates of Aluthge and mean transforms, $$ W_{\exp(\widehat{\alpha}^{(n)})}=\Delta^{(n)}\big(W_{\exp(\alpha)}\big).$$
 where,  $\exp(\beta)=\big(\exp(\beta_i)\big)_{i\in\N}$ for any sequence $\beta=(\beta_i)_{i\in\N}$.

In contrast to what happens with the iterates of the Aluthge transform, the spectrum of the mean transform is not the same as the spectrum of the original operator. \ Moreover, in general the sequence of norms of the mean iterates does not converge to the spectral radius, as shown by the following example. 
 \begin{example}\label{ex1}
 Let $W_\alpha$ be the unilateral weighted shift defined by $\alpha \equiv (\alpha_i)_{i\in\N}$, where  $\alpha_i:=2+(-1)^i$. \ As proven in \cite{hyy}, $\widehat{W_\alpha}=2U_+$ (the unweighted unilateral shift), and is therefore quasi-normal. \ However, $W_\alpha$ is not quasinormal. \ This proves that the inverse mean transform does not preserve the set of quasinormal operators.

On the other hand, by the formula for the spectral radius of a unilateral weighted shift $W_\alpha$ (given in \cite[Problem 91]{Halmos}), we get the following
\begin{eqnarray*}
r(W_\alpha)&=&\underset{n\rightarrow \infty}{\lim}\Big(\underset{k\in\N}{\sup}(\alpha_k...\alpha_{k+n-1})\Big)^{\frac{1}{n}}
\\&=&\underset{n\rightarrow \infty}{\lim}\Big(\underset{k\in\N}{\sup}\underset{k\leq j\leq k+n-1}{\prod}(2+(-1)^j)\Big)^{\frac{1}{n}}
\\&=&\underset{n\rightarrow \infty}{\lim}\Big(\underset{k\in\N}{\sup}\underset{\underset{\text{and}\;j\;\;\text{ peer}\;\;}{k\leq j\leq k+n-1}}{\prod}3 \Big)^{\frac{1}{n}}
\\&=& \underset{n\rightarrow \infty}{\lim} (3^{\frac{n}{2}+\delta})^{\frac{1}{n}}, \;\;(\text{ where $\delta \in \{-1,0,1\}$})
\\&=&\underset{n\rightarrow \infty}{\lim} 3^{\frac{1}{2}+\frac{\delta}{n}}=\sqrt{3}.
\end{eqnarray*}
 Hence, from \cite[Section 4]{Shields}, we conclude that
$$
\sigma(W_\alpha)=\{\lambda \in\C  :  |\lambda|\leq \sqrt{3}\}\ne \sigma(\widehat{W_\alpha})=2\sigma(U_+)=\{\lambda \in \C  :  |\lambda|\leq 2\}.
$$
On the other hand, the mean iterates of $W_\alpha$ are 
 $$\widehat{W_\alpha}^{(n)}=2U_+\;\; \text{for all } \;\;n\in\N.$$
 Therefore $\ell(W_\alpha)=2\ne r(W_\alpha)$. \ Thus, in general the sequence of operator mean iterates does not converge to the spectral radius. \ This is in sharp contrast to what happens for the Aluthge transform (see \cite{yam}).  \qed
\end{example}
 
 \begin{theorem}
 Let $\alpha=(\alpha_i)_{i\in\N}$ be a sequence of strictly positive numbers,  and let  $W_\alpha$ be the associate weighted shift. \ Then 
 $$r(W_\alpha) \leq r(\widehat{W_\alpha}).$$
 \end{theorem}
 \begin{proof}
From the spectral radius formula we obtain
\begin{eqnarray*}
r(W_\alpha)&=&r(\Delta(W_\alpha))
\\&=& \underset{n\rightarrow\infty}{ \lim}\Big(\underset{i\in\N}{\sup} 
 \prod_{j=0}^{n-1}\alpha^{\Delta}_{i+j}  \Big)^{\frac{1}{n}}
 \\&=&  \underset{n\rightarrow\infty}{ \lim}\Big(\underset{i\in\N}{\sup} 
 \prod_{j=0}^{n-1}\sqrt{\alpha_{i+j}\alpha_{i+j+1}} \Big)^{\frac{1}{n}}
 \\&\leq &  \underset{n\rightarrow\infty}{ \lim}\Big(\underset{i\in\N}{\sup} 
 \prod_{j=0}^{n-1}\frac{1}{2}(\alpha_{i+j}+\alpha_{i+j+1}) \Big)^{\frac{1}{n}}\;\;(\sqrt{ab}\leq \frac{1}{2}(a+b) )
 \\&=& \underset{n\rightarrow\infty}{ \lim}\Big(\underset{i\in\N}{\sup} 
 \prod_{j=0}^{n-1}\widehat{\alpha_{i+j}}  \Big)^{\frac{1}{n}}
 \\&=& r(\widehat{W_\alpha}).
\end{eqnarray*}
 \end{proof}
 As a direct consequence, we have the following result.
 \begin{corollary}For a unilateral weighted shift $W_\alpha$, we have 
 $$\sigma(W_\alpha)\subseteq \sigma(\widehat{W_\alpha}).$$
 \end{corollary}
 
\begin{theorem}\label{th6} Let $\alpha=(\alpha_i)_{i\in\N}$ be a sequence of strictly positive numbers, and let $\beta:=\exp(\alpha)$. \ The following estimate for the mean limit of $W_\alpha$ holds:
$$r(W_\alpha)\leq \log(r(W_\beta))= \ell(W_\alpha).$$
\end{theorem} 

\begin{proof} Using the iterates of Aluthge and mean transforms for the weighted shift $W_{\alpha}$ and  Yamazaki's formula for the spectral radius (via the iterates of the Aluthge transform), we get
 \begin{eqnarray} \label{neweq}
  r(W_\beta)&=&\underset{n\rightarrow\infty}{ \lim}\|\Delta^{(n)}(W_\beta)\| 
 =\underset{n\rightarrow\infty}{ \lim}\underset{i\in\N}{\sup}\Big(\prod_{j=0}^{n}\beta_{i+j} ^{(_j ^n)} \Big)^{\frac{1}{2^n}}\;\;(\text{Yamazaki's formula}).
 \end{eqnarray}
Using  (\ref{neweq}) and the particular form of the mean iterates of a unilateral weighted shift (cf. (\ref{eqm})), we obtain
 \begin{eqnarray*}
\ell(W_\alpha)&=& \underset{n\rightarrow\infty}{ \lim} \big\|\widehat{W_\alpha}^{(n)}\big\|
\\&=& \underset{n\rightarrow\infty}{ \lim}\underset{i\in\N}{\sup} \dfrac{\sum_{i=0}^n(^n_j)\alpha_{i+j}}{2^n}
\\&=&\underset{n\rightarrow\infty}{ \lim}\underset{i\in\N}{\sup} \dfrac{\sum_{i=0}^n(^n_j) \log(\beta_{i+j})}{2^n}\;\;\Big( \text{ recall that  }\;\;\log(\beta_k)=\alpha_k, \text{ for all }   k\in\N \Big)
\\&=&\log\Big(\underset{n\rightarrow\infty}{ \lim}\underset{i\in\N}{\sup}\Big(\prod_{j=0}^{n}\beta_{i+j} ^{(_j ^n)} \Big)^{\frac{1}{2^n}} \Big)
\\&=&\log\Big(\underset{n\rightarrow\infty}{ \lim}\big\|\Delta^{(n)}(W_\beta)\big\| \Big)
\\&=& \log\big(r(W_\beta)\big)~~~~~~~~~~~~~~~\;\;\; \;\;(\text{Yamazaki's formula for the spectral radius \cite{yam}})
\end{eqnarray*}

 \begin{eqnarray*}
&=&\log\Big(\underset{n\rightarrow\infty}{ \lim}\Big(\underset{i\in\N}{\sup}\prod_{j=0}^{n-1}\beta_{i+j}  \Big)^{\frac{1}{n}} \Big)  \;\;(\text{spectral radius formula for weighted shifts \cite{Shields}})
 \\&=&\underset{n\rightarrow\infty}{ \lim}\Big(\underset{i\in\N}{\sup}\frac{1}{n}\log\Big(\prod_{j=0}^{n-1}\beta_{i+j}  \Big) \Big)
=\underset{n\rightarrow\infty}{ \lim}\Big(\underset{i\in\N}{\sup}\frac{1}{n}\sum_{j=0}^{n-1}\log(\beta_{i+j})\Big)
 \\&=&\underset{n\rightarrow\infty}{ \lim}\Big(\underset{i\in\N}{\sup}\frac{1}{n}\sum_{j=0}^{n-1}\alpha_{i+j}\Big)
 \\&\geq & \underset{n\rightarrow\infty}{ \lim}\Big(\underset{i\in\N}{\sup} 
 \prod_{j=0}^{n-1}\alpha_{i+j}  \Big)^{\frac{1}{n}} \;\;\big(\text{using the arithmetic-geometric mean inequality} \big)
 \\&=& r(W_\alpha). 
\end{eqnarray*}  
\end{proof}

\begin{remark}
In general the inequality in Theorem \ref{th6} can be strict, as shown in Example \ref{ex1}.
\end{remark}
  
On the other hand, when the sequence $(\alpha_i)_{i\in\N}$ converges we have the following. 

\begin{proposition} \label{prop4}
Let $\alpha=(\alpha_i)_{i\in\N}$ be a sequence of positive numbers ($\alpha_i>0$) and assume that {\color{red}$\alpha$} converges. \ Then 
$$\ell(W_\alpha)=r(W_\alpha)=\underset{i\rightarrow \infty}{\lim}~~\alpha_i .$$
\end{proposition}

\begin{proof}We let $r_0:=r(W_\alpha)=\underset{i\rightarrow \infty}{\lim}~~\alpha_i$. \ Then, for every $\epsilon>0$ there exists $n_0\in\N$ such that $|\alpha_i-r_0|\leq {\epsilon}$ for all  $i\geq n_0$. 

On the other hand,  
$$
\dfrac{(^n_j)}{2^n}\leq K \dfrac{n^{j+1}}{2^n}\underset{n\rightarrow \infty}{\longrightarrow} 0\;\; \text{for any }\; 0\leq i\leq n_0 , \; \text{where $K$ is a fixed constant}. 
$$
Then 
$$
\dfrac{\sum_{j=0}^{n_0}(^n_j)}{2^n}\underset{n\rightarrow \infty}{\longrightarrow} 0.
$$
Hence, there exists $n_1>n_0$, such that 
$$\dfrac{\sum_{j=0}^{n_0}(^n_j)}{2^n}\leq \epsilon, \;\;\text{for all}\;\; n\geq n_1.$$
On the other hand, for $n>n_1$ we have 
\begin{eqnarray*}
\Big|\big\|\widehat{W_\alpha}^{(n)}\big\|-r_0\Big| &=& \Big|\sup_ {i\in\N}~ \alpha_i^{(n)}-r_0\Big|\\ &\leq & \sup_{i\in\N}\big|\alpha_i^{(n)}-r_0\big|\\
&=&  \sup_{i\in\N}\Big|\dfrac{\sum_{j=0}^n(^n_j)\alpha_{i+j}}{2^n}-r_0\Big|
\\ &=& \sup_{i\in\N}\Big|\dfrac{\sum_{j=0}^n(^n_j)(\alpha_{i+j}-r_0)}{2^n}\Big|
\\&\leq &  \sup_{i\in\N}\dfrac{\sum_{j=0}^n(^n_j)|\alpha_{i+j}-r_0|}{2^n}
\\&=&  \sup_{i\in\N}\underset{\leq M \epsilon }{\underbrace{\dfrac{\sum_{j=0}^{n_0}(^n_j)|\alpha_{i+j}-r_0|}{2^n}}}+\underset{\leq \dfrac{\sum_{j=n_0+1}^{n}(^n_j)~~\epsilon}{2^{n+1}}\leq  {\epsilon}}{\underbrace{\dfrac{\sum_{j=n_0+1}^{n}(^n_j)|\alpha_{i+j}-r_0|}{2^n}}}
\\&\leq & \epsilon(M+1),
\end{eqnarray*}
where $M:=\underset{i\in\N}{\sup} |\alpha_i-r_0|$. \ This completes the proof. 
\end{proof}

\begin{remark}
Any semi-hyponormal unilateral weighted shift has a weight sequence satisfying the hypothesis of Proposition \ref{prop4}.
\end{remark}

{\bf Acknowledgments.} The authors would like to thank the referee for carefully reading our manuscript and making many valuable suggestions.

\end{document}